\documentclass{commat}

\newcommand{\Der}{\mathrm{Der}}
\newcommand{\Com}{\mathrm{Com}}
\newcommand{\Lie}{\mathrm{Lie}}
\newcommand{\Pois}{\mathrm{Pois}}

\newcommand{\As}{\mathrm{As}}
\newcommand{\SGD}{\mathrm{SGD}}
\newcommand{\GD}{\mathrm{GD}}

\DeclareSymbolFont{cyrletters}{OT2}{wncyr}{m}{n}
\DeclareMathSymbol{\Sha}{\mathalpha}{cyrletters}{"58}

\title{%
    Some generalizations of the variety of transposed Poisson algebras 
    }

\author{%
    Bauyrzhan Sartayev
    }

\authorinfo[%
    B. Sartayev]{
    SDU University, Kaskelen, Kazakhstan and Narxoz University, Almaty, Kazakhstan}{%
    baurjai@gmail.com
    }

\abstract{%
    It is shown that the variety of transposed Poisson algebras coincides with the variety of Gelfand-Dorfman algebras in which the Novikov multiplication is commutative. The Gr\"obner-Shirshov basis for the transposed Poisson operad is calculated up to degree 4. Furthermore, we demonstrate that every transposed Poisson algebra is F-manifold. We verify that the special identities of GD-algebras hold in transposed Poisson algebras. Finally, we propose a conjecture stating that every transposed Poisson algebra is special, i.e., can be embedded into a differential Poisson algebra.
    }



\VOLUME{32}
\YEAR{2024}
\NUMBER{2}
\firstpage{55}
\DOI{https://doi.org/10.46298/cm.11346}
\begin{document}

\section{Introduction}

A vector space~$A$ with a bilinear product~$\circ$ satisfying the identities
\begin{gather}
(x_1\circ x_2)\circ x_3-x_1\circ(x_2\circ x_3)=(x_2\circ x_1)\circ x_3-x_2\circ(x_1\circ x_3), \label{LeftSym} \\
(x_1\circ x_2)\circ x_3=(x_1\circ x_3)\circ x_2, \label{RightCom}
\end{gather}
is called a Novikov algebra.
A linear space $V$ with two bilinear operations $ \circ $ and $[\cdot,\cdot]$ is called 
a {\em Gelfand--Dorfman algebra} (or simply $\GD$-algebra) \cite{Xu2000}, \cite{Xu2002} if $(V,\circ)$ is a 
Novikov algebra and $(V,[\cdot,\cdot])$ is a Lie algebra, with the following 
additional identity:

\begin{gather}\label{eq:GD1}
[x_1,x_2\circ x_3]-[x_3,x_2\circ x_1]+[x_2,x_1]\circ x_3-[x_2,x_3]\circ x_1-x_2\circ [x_1,x_3]=0.
\end{gather}

Let us provide the motivation for considering these algebras. The variety of Gelfand-Dorfman algebras is a natural generalization of the variety of Novikov algebras, while the variety of  Poisson algebras is a generalization of the variety of commutative associative algebras. It is well-known that a free Novikov algebra can be embedded into a free commutative associative algebra with a derivation \cite{dzhuma}. It was later discovered \cite{KSO2019} that it is also possible to extend the mapping from free Novikov algebras to differential commutative associative algebras to the mapping from free Gelfand-Dorfman algebras to free Poisson algebras with a derivation, as follows:
\[
 \begin{aligned}
  \tau : & \;\GD \to \Pois\Der, \\
   & x_1\circ x_2\mapsto x_1 d(x_2), \\
   & [x_1,x_2]\mapsto \{x_1,x_2\},
 \end{aligned}
\]
where $\Pois\Der$ is a Poisson algebra with derivation. However, it turns out \cite{KS2022} that the mapping $\tau$ is not always injective, that is, there exist non-embeddable $\GD$-algebras into $\Pois\Der$ algebras with respect to the defined mapping $\tau$. 

We say that a $\GD$-algebra $A$ is {\em special} if it can be embedded into some $\Pois\Der$ algebra with respect to $\tau$.
Denote by $\SGD$ the variety of algebras generated by special $\GD$-algebras, namely, the class of all homomorphic images of all special $\GD$-algebras. 
A polynomial identity is said to be {\it special}  for the variety $\SGD$ if it holds on all special GD-algebras but it does not hold on all  GD-algebras. In \cite{KSO2019}, a list of special identities for the variety $\SGD$ up to degree four was obtained:
\begin{gather}
[x_1,(x_2\circ x_3)\circ x_4]-[x_1,x_2\circ x_3]\circ x_4-[x_1,x_2\circ x_4]\circ x_3+([x_1,x_2]\circ x_3)\circ x_4=0, \label{spec1} 
\end{gather}
\begin{multline}\label{spec2}
[x_4\circ x_1,x_3\circ x_2]-[x_3\circ x_1,x_4\circ x_2]+[x_3,x_4\circ x_1]\circ x_2-[x_4,x_3\circ x_2]\circ x_1 \\ 
-[x_4,x_3\circ x_1]\circ x_2+[x_3,x_4\circ x_2]\circ x_1+2([x_4,x_3]\circ x_1)\circ x_2.
\end{multline}
A basis of a free special $\GD$-algebra was constructed in \cite{KSO2019}. In \cite{GS2023} was constructed the monomial basis of the free special $\GD$-algebra in terms of $\circ$ and $[\cdot,\cdot]$. Recall that a basis of free Novikov algebra was found in \cite{dzhuma} using rooted trees, and the basis of free Lie algebra is the Lyndon-Shirshov words \cite{Shirshov1958}. Therefore, one can easily construct a basis of a free algebra that involves the Novikov multiplication $\circ$ and Lie product $[\cdot,\cdot]$. The problem of constructing a basis becomes more complicated when we consider those multiplications with additional identity (\ref{eq:GD1}).

In this paper, we consider the variety of Gelfand-Dorfman algebras whose Novikov multiplication is commutative, i.e.,
\begin{gather}
x_1\circ x_2=x_2\circ x_1. \label{comGelfand-Dorfman1} 
\end{gather}
We observe that commutative Novikov algebra satisfies the associative identity:
$$(x_1 \circ x_2)\circ x_3=(x_2 \circ x_1)\circ x_3=(x_2 \circ x_3)\circ x_1=x_1 \circ (x_2\circ x_3).$$
We call this algebra commutative $\GD$-algebra. In this case the identity (\ref{eq:GD1}) can be written as
\begin{gather}\label{eq:GD-com}
[x_1,x_2 x_3]-[x_3,x_2 x_1]+[x_2,x_1]x_3-[x_2,x_3]x_1-[x_1,x_3]x_2=0.
\end{gather}
Let us formulate the main results of this paper:
for commutative Gelfand-Dorfman operad, we determine a set of basis elements up to degree 4. Using these obtained basis elements we show that every commutative Gelfand-Dorfman algebra is $F$-manifold, and we show that the variety of commutative $\GD$-algebras coincides with the variety of transposed Poisson algebras, which has garnered considerable attention in recent years \cite{TP1,TP2,TP3,TP4,TP5}.
All mentioned results give
$$\GD/(a\circ b-b\circ a)=\textrm{TP}\subset\textrm{F-manifold},$$
where TP is a variety of transposed Poisson algebras. In other words, in commutative $\GD$-algebra the identity (\ref{eq:GD1}) can be rewritten to
$$2[x_1,x_2]x_3 = [x_1x_3,x_2]+[x_1,x_2x_3],$$
which is defining the identity of the variety of TP-algebras.

In addition, we prove that a free transposed Poisson algebra satisfies special identities (\ref{spec1}) and (\ref{spec2}). It raises a question about the speciality of transposed Poisson algebras but it is still an open question.

In this paper, all algebras are defined over a field of characteristic~0.

\section{Gr\"obner base of commutative GD-operad}

The operad theory is a powerful tool for computing a multilinear basis of fixed algebra. The main method for this calculation is the shuffling of monomials with so-called the forgetful functor $f$. For more details about this functor, see \cite{bremner-dotsenko,DotKhor2010}. 

\begin{definition}\label{shuffle operad}(\cite{bremner-dotsenko})
A shuffle operad is a monoid in the category of nonsymmetric collections of vector spaces with respect to the shuffle composition product defined as follows:
$$
\mathcal{V}\circ_\Sha\mathcal{W}(n)=\bigoplus_{r\geq 
1}\mathcal{V}(r)\otimes\bigoplus_{\pi}\mathcal{W}(|I^{(1)}|)\otimes 
\mathcal{W}(|I^{(2)}|)\otimes\ldots\otimes\mathcal{W}(|I^{(r)}|),
$$
where $\mathcal{W}$ and $\mathcal{V}$ are nonsymmetric collections, $\pi$ ranges in all set partitions $\{1,\ldots,n\}$ $=\bigsqcup_{j=1}^r I^{(j)}$ for which all parts $I^{(j)}$ are nonempty and $\mathrm{min}(I^{(1)})<\ldots<\mathrm{min}(I^{(r)})$.
\end{definition}

To shuffle given monomials, we extend the alphabet of operations until it satisfies the definition of shuffle operad. For the case of operad $\mathcal{A}s$ governed by the variety of associative algebras, we need to add additional operation $y$ for given $x$ as follows:
$$x(2\; 1)=y(1\; 2),$$
and we obtain the following relations of operad $\As_{\Sha}(x,y)$
\begin{multline*}
x(x(1\; 2)\; 3) - x(1\; x(2\; 3)),\;\;\;
x(y(1\; 2)\; 3) - y(x(1\; 3)\; 2),\;\;\;
x(x(1\; 3)\; 2) - x(1\; y(2\; 3)), \\
x(y(1\; 3)\; 2) - y(x(1\; 2)\; 3),\;\;\;
y(1\; x(2\; 3)) - y(y(1\; 3)\; 2),\;\;\;
y(1\; y(2\; 3)) - y(y(1\; 2)\; 3),
\end{multline*}
where $\As_{\Sha}(x,y)$ is a shuffled operad $\As$.
To obtain the defining relations of operad $\Com_{\Sha}(x)$ (associative and commutative) we have to add to the given relations commutative identity:
$$x(1\; 2) - y(1\; 2).$$

In the same way, we write the defining relations of operad $\Lie_{\Sha}(z)$ governed by the variety of Lie algebras:
\begin{equation}\label{eq:Lie}
    z(z(1\; 2)\; 3) - z(1\; z(2\; 3)) - z(z(1\; 3)\; 2).
\end{equation}

For our purposes, we write the defining identities of the operad $\Com\textrm{-}\GD_{\Sha}(x,z)$ governed by the variety of commutative $\GD$-algebras. It remains only to rewrite the identity (\ref{eq:GD-com}):
\begin{multline*}
z(1\; x(2\; 3)) + z(x(1\; 2)\; 3) - x(z(1\; 2)\; 3) - x(1\; z(2\; 3)) - x(z(1\; 3)\; 2),\\
-z(x(1\; 3)\; 2) + z(x(1\; 2)\; 3) + x(z(1\; 2)\; 3) - x(z(1\; 3)\; 2) - x(1\; z(2\; 3)),\\
-x(z(1\; 2)\; 3) + z(1\; x(2\; 3)) + z(x(1\; 3)\; 2) - x(z(1\; 3)\; 2) + x(1\; z(2\; 3)).
\end{multline*}

Calculating the Gr\"obner base of $\Com\textrm{-}\GD_{\Sha}(x,z)$ by means of the package \cite{DotsHij}, 
we get the following result:
\begin{center}
\begin{tabular}{c|cccccc}
 $n$ & 1 & 2 & 3 & 4 & 5 & 6 \\
 \hline 
 $\dim(\Com\textrm{-}\GD(n)) $ & 1 & 2 & 6 & 20 & 74 & 301
\end{tabular}
\end{center}
Moreover, we obtain the following result:
\begin{theorem}\label{Grobner}
The Gr\"obner basis of the operad $\Com\textrm{-}\GD_{\Sha}(x,z)$ up to degree $4$ is defined by the following relations:
\end{theorem}
\begin{gather}
x(x(1\; 3)\; 2)  \rightarrow  x(1\; x(2\; 3)),\;\;\;
x(x(1\; 2)\; 3)  \rightarrow  x(1\; x(2\; 3)),\label{4}\\
z(x(1\; 2)\; 3)  \rightarrow  2\; x(z(1\; 3)\; 2)  -  z(1\; x(2\; 3)),\label{5}\\
x(z(1\; 2)\; 3)  \rightarrow  x(z(1\; 3)\; 2)  -  x(1\; z(2\; 3)),\label{6}\\
z(z(1\; 2)\; 3)  \rightarrow  z(z(1\; 3)\; 2)  +  z(1\; z(2\; 3)),\label{7}\\
z(x(1\; 3)\; 2)  \rightarrow  2\; x(z(1\; 3)\; 2)  -  z(1\; x(2\; 3))  -  2\; x(1\; z(2\; 3)),\label{8}\\
x(z(1\; 4)\; z(2\; 3))  \rightarrow  x(z(1\; 3)\; z(2\; 4))  -  x(z(1\; z(3\; 4))\; 2)  +  x(1\; z(2\; z(3\; 4))),\label{9}
\end{gather}
\begin{multline}\label{10}
x(z(1\; 3)\; x(2\; 4))  \rightarrow  z(1\; x(2\; x(3\; 4)))  +  3\; x(1\; x(z(2\; 4)\; 3)) -  2\; x(1\; z(2\; x(3\; 4))) \\
-  2\; x(1\; x(2\; z(3\; 4))),
\end{multline}
\begin{multline}\label{11}
x(z(1\; 4)\; x(2\; 3))  \rightarrow  z(1\; x(2\; x(3\; 4)))  +  3\; x(1\; x(z(2\; 4)\; 3))  -  2\; x(1\; z(2\; x(3\; 4))) \\  -  x(1\; x(2\; z(3\; 4))),
\end{multline}
\begin{multline}\label{12}
x(z(1\; x(3\; 4))\; 2)  \rightarrow  z(1\; x(2\; x(3\; 4)))  +  2\; x(1\; x(z(2\; 4)\; 3))  -  x(1\; z(2\; x(3\; 4))) \\ -  x(1\; x(2\; z(3\; 4))),
\end{multline}
\begin{multline}\label{13}
z(z(1\; 4)\; x(2\; 3))  \rightarrow  z(z(1\; x(3\; 4))\; 2)  -  2\; x(z(1\; 3)\; z(2\; 4))  +  z(1\; x(z(2\; 4)\; 3)),\\
+  z(1\; z(2\; x(3\; 4)))  +  2\; x(1\; z(z(2\; 4)\; 3)),
\end{multline}
\begin{multline}\label{14}
z(z(1\; 3)\; x(2\; 4))  \rightarrow  z(z(1\; x(3\; 4))\; 2)  -  2\; x(z(1\; 3)\; z(2\; 4))  +  2\; x(z(1\; z(3\; 4))\; 2)\\
+  z(1\; x(z(2\; 4)\; 3))+  z(1\; z(2\; x(3\; 4)))  -  z(1\; x(2\; z(3\; 4)))  +  2\; x(1\; z(z(2\; 4)\; 3))
\end{multline}
The rewriting rule (\ref{5}) coincides with the identity of transposed Poisson algebra which gives $\Com\textrm{-}\GD\subseteq\textrm{TP}$. Analogically, calculating the Gr\"obner base of $\textrm{TP}_{\Sha}(x,z)$, we obtain the same result as for $\Com\textrm{-}\GD_{\Sha}(x,z)$ which means $\Com$-$\GD=\textrm{TP}$.

\section{Examples of transposed Poisson algebras}

Let us turn our attention to an important class of algebras which is called $F$-manifold.
$F$-manifold algebras appear in many fields of mathematics such as singularity
theory \cite{22}, quantum $\mathbb{K}$-theory \cite{25}, integrable systems \cite{10,11,30}, operad \cite{33}, algebra \cite{40}.

\begin{definition}
A $F$-$manifold$ algebra is a triple $(\cdot,[\cdot,\cdot],X)$, where the multiplication $\cdot$ is associative and commutative, and $[\cdot,\cdot]$ is a Lie bracket with the additional identity
\begin{multline}\label{eq:manifold}
    [a_1\cdot a_2,a_3 \cdot a_4] = [a_1 \cdot a_2,a_3] \cdot a_4 +[a_1 \cdot a_2,a_4] \cdot a_3 + a_1 \cdot [a_2,a_3 \cdot a_4]+ a_2 \cdot [a_1,a_3 \cdot a_4]- \\
(a_1 \cdot a_3) \cdot [a_2,a_4]-(a_2 \cdot a_3) \cdot [a_1,a_4]-(a_2 \cdot a_4) \cdot [a_1,a_3]-(a_1 \cdot a_4) \cdot [a_2,a_3].
\end{multline}
\end{definition}

\begin{theorem}
Every transposed Poisson algebra is $F$-manifold.
\end{theorem}
\begin{proof}
Since both multiplications are the same, it is enough to prove that the identity $(\ref{eq:manifold})$ holds in the free transposed Poisson algebra. For that purpose, we use the rewriting system of Theorem \ref{Grobner} to the identity (\ref{eq:manifold}). For simplicity, we write $ab$ for $a\cdot b$ for all $a$ and $b$.
Initially, let us start from the monomials $[a_1 a_2,a_3  a_4]$, $[a_1a_2,a_3]a_4$, $[a_1a_2,a_4]a_3$, $[a_1,a_3a_4]a_2$ and $a_1[a_2,a_3]a_4$:
\begin{multline*}
    [a_1 a_2,a_3 a_4]=^{(\ref{5})}2[a_1,a_3a_4]a_2-[a_1,a_2a_3a_4]=^{(\ref{12})}\\
    [a_1,a_2a_3a_4]+4a_1[a_2,a_4]a_3-2a_1[a_2,a_3a_4]-2a_1a_2[a_3,a_4],
\end{multline*}
\begin{multline*}
    -[a_1a_2,a_3]a_4=^{(\ref{5})}-2[a_1,a_3]a_2a_4+[a_1,a_2a_3]a_4=^{(\ref{6}),(\ref{10})}-[a_1,a_3]a_2a_4\\
    -[a_1,a_2a_3a_4]-3a_1[a_2,a_4]a_3+2a_1[a_2,a_3a_4]+2a_1a_2[a_3,a_4]+[a_1,a_4]a_2a_3-a_1[a_2a_3,a_4]\\
    =^{(\ref{5})}-[a_1,a_3]a_2a_4-[a_1,a_2a_3a_4]-3a_1[a_2,a_4]a_3+2a_1[a_2,a_3a_4]+2a_1a_2[a_3,a_4]\\
    +[a_1,a_4]a_2a_3-2a_1[a_2,a_4]a_3+a_1[a_2,a_3a_4],
\end{multline*}
\begin{multline*}
    -[a_1a_2,a_4]a_3=^{(\ref{5})}-2[a_1,a_4]a_2a_3+[a_1,a_2a_4]a_3=^{(\ref{6})}-2[a_1,a_4]a_2a_3+ [a_1,a_3]a_2a_4\\
    -a_1[a_2a_4,a_3]=^{(\ref{10}),(\ref{8})}-2[a_1,a_4]a_2a_3+[a_1,a_2a_3a_4]+a_1[a_2,a_4]a_3-a_1[a_2,a_3a_4],
\end{multline*}
\begin{multline*}
    -[a_1,a_3a_4]a_2=^{(\ref{12})}-[a_1,a_2a_3a_4]
    -2a_1[a_2,a_4]a_3+a_1[a_2,a_3a_4]+a_1a_2[a_3,a_4],
\end{multline*}
$$a_1[a_2,a_3]a_4=^{(\ref{6})}a_1[a_2,a_4]a_3-a_1a_2[a_3,a_4].$$
Substituting the monomials by given expressions we obtain that every TP-algebra satisfies (\ref{eq:manifold}) identity.
\end{proof} 


\section{Special identities of GD-algebras on transposed Poisson algebras}

In \cite{KS2022} proved that the class of special $\GD$-algebras forms a variety. Since every variety can be defined by a set of identities, all known special identities of $\GD$-algebra are contained in this set. There is no special identity of degree $3$, but there are two special identities (\ref{spec1}) and (\ref{spec2}) of degree $4$. Furthermore, in \cite{KS2022}, all special identities of degree $5$ were computed, and it was demonstrated that they are a consequence of the identities (\ref{spec1}) and (\ref{spec2}).

\begin{theorem}
Every TP-algebra satisfies the identities $(\ref{spec1})$ and $(\ref{spec2})$.
\end{theorem}
\begin{proof}
Using the rewriting system of Theorem \ref{Grobner} let us first rewrite all monomials in the identity (\ref{spec1}) and then consider their sum.
\begin{multline*}
    [a_1,a_2]a_3a_4=^{(\ref{6})}[a_1,a_3a_4]a_2-a_1[a_2,a_3a_4]=^{(\ref{12})}[a_1,a_2a_3a_4]\\
    +2a_1[a_2,a_4]a_3-a_1[a_2,a_3a_4]-a_1a_2[a_3,a_4]-a_1[a_2,a_3a_4],
\end{multline*}
\begin{multline*}
    -[a_1,a_2a_3]a_4=^{(\ref{6})}-[a_1,a_4]a_2a_3+a_1[a_2a_3,a_4]=^{(\ref{11}),(\ref{5})}-[a_1,a_2a_3a_4]\\
    -3a_1[a_2,a_4]a_3+2a_1[a_2,a_3a_4]+a_1a_2[a_3,a_4]+2a_1[a_2,a_4]a_3-a_1[a_2,a_3a_4],
\end{multline*}
\begin{multline*}
    -[a_1,a_2a_4]a_3=^{(\ref{6})}-[a_1,a_3]a_2a_4+a_1[a_2a_4,a_3]=^{(\ref{10}),(\ref{8})}-[a_1,a_2a_3a_4]\\
    -3a_1[a_2,a_4]a_3+2a_1[a_2,a_3a_4]+2a_1a_2[a_3,a_4]+2a_1[a_2,a_4]a_3-a_1[a_2,a_3a_4]-2a_1a_2[a_3,a_4].
\end{multline*}
We note that the sum of the monomials gives zero and therefore, the identity (\ref{spec1}) is an identity in a TP-algebra.

In a similar way, we show that the identity (\ref{spec2}) is an identity in a TP-algebra. 
\begin{multline*}
    [a_1a_3,a_2a_4]-[a_1a_4,a_2a_3]=^{(\ref{8})} 2[a_1,a_3]a_2a_4-2a_1[a_2a_4,a_3]-2[a_1,a_4]a_2a_3\\
    +2a_1[a_2a_3,a_4]=^{(\ref{10}),(\ref{11})}-2a_1a_2[a_3,a_4]-2a_1[a_2a_4,a_3]+2a_1[a_2a_3,a_4],
\end{multline*}
\[
[a_1a_4,a_3]a_2=^{(\ref{8})}2[a_1,a_4]a_3a_2-[a_1,a_3a_4]a_2-2a_1[a_3,a_4]a_2,
\]
\[
-[a_1a_3,a_4]a_2=^{(\ref{5})}-2[a_1,a_4]a_3a_2+[a_1,a_3a_4]a_2.
\]
\end{proof}

In \cite{KolPan2019} is given an example of exceptional $\GD$-algebra. However, it was proved that Novikov commutator algebras are special, namely, every algebra from this class can be embedded into appropriate Poisson algebra with derivation. As we saw since the variety of transposed Poisson algebras contained in the variety of $\GD$-algebras and every transposed Poisson algebra satisfy special identities (\ref{spec1})  and (\ref{spec2}), therefore, we have the following conjecture
\begin{conjecture}
Every transposed Poisson algebra is special.
\end{conjecture}

In \cite{Kol-Nes} was given an answer for a particular case of the conjecture. In a related study \cite{PK-BS}, the embedding of left-symmetric algebras into differential perm algebras was explored, and a criterion for such embedding was discovered. In the event of a negative answer, it may be worth exploring the possibility of establishing a similar test for the present case.

\begin{center} ACKNOWLEDGMENTS   \end{center}
The author is grateful to the referees for their valuable remarks that improved the exposition. The author is also grateful to Professor N. Ismailov for his comments and advice, thanks to which the article has acquired such a full-fledged look.

This research was funded by the Science Committee of the Ministry of Science and Higher Education of the Republic of Kazakhstan (Grant No. AP14871710).



{\small\bibliography{commat}}

\begin{thebibliography}{10}

\bibitem{TP1}
C.~M. Bai, R.~P. Bai, L.~Guo, and Y.~Wu.
\newblock Transposed {P}oisson algebras, {N}ovikov-{P}oisson algebras and 3-{L}ie algebras.
\newblock {\em Journal of Algebra}, 632:535--566, 2023.

\bibitem{TP2}
P.~D. Beites, A.~F. Ouaridi, and I.~Kaygorodov.
\newblock The algebraic and geometric classification of transposed {P}oisson algebras.
\newblock {\em Revista de la Real Academia de Ciencias Exactas, Físicas y Naturales. Serie A: Matemáticas. RACSAM}, 117(2, Paper No. 55):25 pp, 2023.

\bibitem{bremner-dotsenko}
M.~R. Bremner and V.~Dotsenko.
\newblock {\em Algebraic {O}perads {A}n {A}lgorithmic {C}ompanion}.
\newblock Chapman Hall, 2016.

\bibitem{10}
L.~David and I.~A.~B. Strachan.
\newblock Compatible metrics on a manifold and nonlocal bi-{H}amiltonian structures.
\newblock {\em International Mathematics Research Notices}, 66:3533--3557, 2004.

\bibitem{11}
L.~David and I.~A.~B. Strachan.
\newblock Dubrovins duality for {F}-manifolds with eventual identities.
\newblock {\em Advances in Mathematics}, 226(5):4031--4060, 2011.

\bibitem{DotsHij}
V.~Dotsenko and W.~Heijltjes.
\newblock Gr\"obner bases for operads, http://irma.math.unistra.fr/~dotsenko/operads.html, 2019.

\bibitem{DotKhor2010}
V.~Dotsenko and A.~Khoroshkin.
\newblock Gr\"obner bases for operads.
\newblock {\em Duke Mathematical Journal}, 153(2):363--396, 2010.

\bibitem{dzhuma}
A.~S. Dzhumadil'daev and C.~L\"ofwall.
\newblock Trees, free right-symmetric algebras, free {N}ovikov algebras and identities.
\newblock {\em Homology, Homotopy And Applications}, 4(2, part 1):165--190, 2002.

\bibitem{GS2023}
V.~Gubarev and B.~K. Sartayev.
\newblock Free special {G}elfand-{D}orfman algebra.
\newblock {\em Journal of Algebra and Its Applications}, doi.org/10.1142/S0219498825500057, 2023.

\bibitem{22}
C.~Hertling.
\newblock {\em Frobenius {M}anifolds and {M}oduli {S}paces for {S}ingularities}.
\newblock Cambridge Tracts in Mathematics. Cambridge University Press, 2002.

\bibitem{TP4}
I.~Kaygorodov and M.~Khrypchenko.
\newblock Transposed {P}oisson structures on {B}lock {L}ie algebras and superalgebras.
\newblock {\em Linear Algebra and Its Applications}, 656:167--197, 2023.

\bibitem{TP3}
I.~Kaygorodov and M.~Khrypchenko.
\newblock Transposed {P}oisson structures on {W}itt type algebras.
\newblock {\em Linear Algebra and Its Applications}, 665:196--210, 2023.

\bibitem{TP5}
I.~Kaygorodov, V.~Lopatkin, and Z.~Zhang.
\newblock Transposed {P}oisson structures on {G}alilean and solvable {L}ie algebras.
\newblock {\em Journal of Geometry and Physics}, 187(Paper No. 104781):13 pp, 2023.

\bibitem{Kol-Nes}
P.~S. Kolesnikov and A.~A. Nesterenko.
\newblock Conformal envelopes of {N}ovikov–{P}oisson algebras.
\newblock {\em Siberian Mathematical Journal}, 64(3):598--610, 2023.

\bibitem{KolPan2019}
P.~S. Kolesnikov and A.~S. Panasenko.
\newblock Novikov commutator algebras are special.
\newblock {\em Algebra and Logic}, 58(6):538--539, 2020.

\bibitem{KS2022}
P.~S. Kolesnikov and B.~Sartayev.
\newblock On the special identities of {G}elfand-{D}orfman algebras.
\newblock {\em Experimental Mathematics}, DOI:10.1080/10586458.2022.2041134, 2022.

\bibitem{KSO2019}
P.~S. Kolesnikov, B.~Sartayev, and A.~Orazgaliev.
\newblock {G}elfand-{D}orfman algebras, derived identities, and the {M}anin product of operads.
\newblock {\em Journal of Algebra}, 539:260--284, 2019.

\bibitem{PK-BS}
P.~S. Kolesnikov and B.~K. Sartayev.
\newblock On the embedding of left-symmetric algebras into differential {P}erm algebras.
\newblock {\em Communications in Algebra}, 50(8):3246--3260, 2022.

\bibitem{25}
Y.~P. Lee.
\newblock Quantum {K}-theory, {I}: {F}oundations.
\newblock {\em Duke Mathematical Journal}, 121(3):389--424, 2004.

\bibitem{40}
J.~Liu, Y.~Sheng, and C.~Bai.
\newblock F-{M}anifold algebras and deformation quantization via pre-{L}ie algebras.
\newblock {\em Journal of Algebra}, 559:467--495, 2020.

\bibitem{30}
P.~Lorenzoni, M.~Pedroni, and A.~Raimondo.
\newblock F-manifolds and integrable systems of hydrodynamic type.
\newblock {\em Archivum Mathematicum (Brnö)}, 47(3):163--180, 2011.

\bibitem{33}
S.~A. Merkulov.
\newblock Operads, deformation theory and {F}-manifolds.
\newblock {\em Aspects of Mathematics}, 36:213--251, 2004.

\bibitem{Shirshov1958}
A.~I. Shirshov.
\newblock On free {L}ie rings.
\newblock {\em Matematicheskii Sbornik}, 42(2):113--122, 1958.

\bibitem{Xu2000}
X.~Xu.
\newblock Quadratic {C}onformal {S}uperalgebras.
\newblock {\em Journal of Algebra}, 231(1):1--38, 2000.

\bibitem{Xu2002}
X.~Xu.
\newblock {G}el'fand-{D}orfman bialgebras.
\newblock {\em Southeast Asian Bulletin of Mathematics}, 27:561--574, 2003.

\end{thebibliography}
\EditInfo{May 23, 2023}{September 20, 2023}{Ivan Kaygorodov, Adam Chapman, Mohamed Elhamdadi}
\end{document}